\documentclass[12pt,a4paper]{article}
\usepackage{latexsym,amsfonts,amsmath,graphics}
\usepackage{amsthm}
\usepackage{epsfig}

\newtheorem{theorem}{Theorem}
\newtheorem{lemma}{Lemma}

\newtheorem{example}{Example}

\newtheorem{remark}{Remark}

\usepackage[usenames]{xcolor}
\newif\ifmarkauthors
\markauthorstrue
\ifmarkauthors
  \definecolor{darkred}{RGB}{139,0,0}
  \definecolor{darkgreen}{RGB}{0,100,0}
  \definecolor{darkmagenta}{RGB}{139,0,139}
  \definecolor{darkorange}{RGB}{190,70,20}

  \def\cbdelete[#1]{}
\else

  \def\cbdelete[#1]{}
\fi

\newcommand{\bsgamma}{\boldsymbol{\gamma}}

\newcommand{\bsj}{\boldsymbol{j}}

\newcommand{\cS}{{\cal S}}
\newcommand{\e}{{\varepsilon}}

\newcommand{\uu}{\mathfrak{u}}

\newcommand{\NN}{\mathbb{N}}

\newcommand{\RR}{\mathbb{R}}

\allowdisplaybreaks

\begin{document}

\title{Exponential tractability \\ of linear weighted tensor 
product problems\\
in the worst-case setting  \\
for arbitrary linear functionals
}
\author{Peter Kritzer\thanks{P.~Kritzer
is supported by the Austrian Science Fund (FWF),
Project F5506-N26, which is part of the Special Research 
Program ``Quasi-Monte Carlo Methods:
Theory and Applications''.}, 
Friedrich Pillichshammer\thanks{F.~Pillichshammer
is supported by the Austrian Science Fund (FWF),
Project F5509-N26, which is part of the Special 
Research Program ``Quasi-Monte Carlo Methods:
Theory and Applications''.}, 
Henryk Wo\'zniakowski\thanks{
H. Wo\'zniakowski is supported in part by the National Science Center,
Poland, based on the decision DEC-2017/25/B/ST1/00945.
\newline
The authors thank the Johann Radon Institute for Computational and Applied Mathematics in Linz, 
where most research for this paper has been done during the Special Semester on Multivariate Algorithms and their Foundations in Number Theory, October 15 - December 21, 2018}}
\maketitle

\begin{abstract}
We study the approximation of compact 
linear operators defined over certain weighted tensor product Hilbert spaces. 
The information complexity is defined as the minimal number of arbitrary
linear functionals which is needed to obtain an $\e$-approximation for the
$d$-variate problem. It is fully determined in terms of the
weights and univariate singular values. 
Exponential tractability means that
the information complexity is bounded by a certain function which
depends polynomially on $d$ and logarithmically on $\e^{-1}$.  
The corresponding un-weighted problem was
studied in \cite{HKW19} with many negative results for
exponential tractability.
The product weights studied in the present paper change the
situation. Depending on the form of polynomial dependence on $d$ and
logarithmic dependence on $\e^{-1}$, we  
study exponential strong polynomial, exponential
polynomial, exponential quasi-polynomial, and
exponential $(s,t)$-weak tractability with $\max(s,t)\ge1$.   
{}For all these notions of exponential tractability,
we establish necessary and sufficient conditions on weights
and univariate singular values
for which it is indeed possible to achieve the corresponding notion of
exponential tractability.
The case of exponential $(s,t)$-weak tractability with $\max(s,t)<1$ is
left for future study. 
The paper uses some general results obtained in \cite{HKW19} 
and \cite{KW19}.
\end{abstract}

\centerline{\begin{minipage}[hc]{130mm}{
 }
\end{minipage}}

\section{Introduction}
This paper deals with tractability of linear multivariate problems, which has been studied in a large number of papers, 
and is at the core of the research field of Information-Based Complexity (IBC). For introductions to IBC, we refer to 
the books \cite{TWW} and \cite{TW98}. For a recent and comprehensive overview of results on tractability, we refer the 
interested reader to the trilogy \cite{NW08,NW10,NW12}. 

In the present paper we study the information 
complexity of a compact linear operator $S_d$ 
from a separable Hilbert space $H_d$ into 
another Hilbert space $G_d$. The information complexity 
$n(\varepsilon,S_d)$ is defined as the minimal number of 
linear functionals needed by an algorithm 
which approximates $S_d$ to within an error threshold of
$\varepsilon>0$. 
As shown in \cite{HNW13}, without loss of generality we may consider
only continuous linear functionals from 
the class $\Lambda_d^{\rm all}=
H_d^*$.

We use different notions of tractability to describe how 
the information complexity of a given problem 
depends on $\varepsilon^{-1}$ and $d$ as $\max(\varepsilon^{-1},d)$
tends to infinity in an arbitrary way.  
A problem is called intractable if its information complexity depends 
exponentially on $\varepsilon^{-1}$ or $d$. If a problem is tractable, 
we describe sub-exponential dependence on some powers of
$\varepsilon^{-1}$ and~$d$ by 
using the classification into various notions of tractability, 
which can be summarized by the algebraic (abbreviated ALG) and 
exponential (abbreviated EXP) cases. For the algebraic 
case, we need to verify that $n(\varepsilon, S_d)$ is bounded 
by certain functions of $d$ and $\varepsilon^{-1}$ which are, in particular, 
not exponential in some powers of $d$ and $\varepsilon^{-1}$. 
For the exponential case, we replace 
$\varepsilon^{-1}$ by $1+\log\varepsilon^{-1}$, 
and consider the same notions of tractability as for the algebraic case. 

Most papers on tractability have dealt with  the notions of 
algebraic tractability, which can be said to be the ``standard'' case 
of tractability. 
For an overview of results we again refer to \cite{NW08,NW10,NW12} 
and the references therein, as well as \cite{PP09} and \cite{WW17}. 

However, there is also a recent stream of work on 
exponential tractability, and this is what we are going to study
here.  For results on exponential tractability 
for general linear problems (without necessarily assuming tensor
product structure), we refer the reader 
to the recent paper \cite{KW19}, for exponential 
tractability for linear problems on un-weighted tensor product spaces, 
we refer the reader to \cite{HKW19,PP14,PPW17}. 
Further results on exponential tractability 
can, e.g., be found in the papers \cite{DKPW14,DLPW11,
IKLP15,IKPW16a,IKPW16b,KPW14,KPW17,LX16a,LX16b,LX17,LZ19,
PPXY19,SW18,S17,W19,X15}.

For un-weighted tensor product problems, such as studied in  
\cite{HKW19}, we have many negative results for
exponential tractability. Our point of departure is to verify
if these negative results can be changed if we switch to 
weighted tensor products with 
product weights. Indeed, this is the case. We now
illustrate a sample of our results leaving the general results to
Section~\ref{Results}.

For weighted tensor product problems, the information complexity
depends on two non-increasing sequences, $\{\lambda_j\}_{j\in \NN}$ and
$\{\gamma_j\}_{j\in \NN}$. Here, the $\lambda_j$'s are the squares of 
the ordered singular values of the univariate operator,
and the $\gamma_j$'s are product weights which moderate the importance of the
$j^{{\rm th}}$ univariate problem in the definition of the tensor product 
for the multivariate case.  
Without loss of generality we may assume that $\lambda_1=1$ 
and, to omit the trivial case, that $\lambda_2>0$. To make the
presentation of our results easier we also assume that
all weights $\gamma_j$ are positive. The un-weighted case is
obtained if we take $\gamma_j=1$ for all $j\in\NN$. 

The concept of exponential strong polynomial tractability (EXP-SPT) is defined
when there are two non-negative numbers $C$ and $p$ such that
the information complexity is bounded by
$$
C(1+\log\,\varepsilon^{-1})^p\ \ \ \ \ 
\mbox{for all $\varepsilon\in(0,1]$ and
$d\in\NN$}. 
$$
Similarly, the concept of exponential polynomial tractability (EXP-PT)
is defined 
when there are three non-negative numbers $C,q$, and $p$ such that
the information complexity is bounded by
$$
C\,d^{\,q}(1+\log\,\varepsilon^{-1})^p\ \ \  \mbox{for all 
$\varepsilon\in(0,1]$ and $d\in\NN$}. 
$$
Note that especially EXP-SPT is a quite demanding property 
since the upper bound on the information complexity must be independent 
of $d$ and at most polynomially dependent 
on $1+\log\,\e^{-1}$. For EXP-PT we allow that the information complexity may 
depend polynomially on $d$ and $1+\log\,\e^{-1}$. One might suspect
that EXP-SPT and also EXP-PT hold only for extremely small
$\lambda_j$'s and $\gamma_j$'s. As we shall see, this is indeed the
case.  
 
We prove that 
\begin{itemize}
\item EXP-SPT and EXP-PT are equivalent,
\item EXP-SPT holds if and only if
$$
\lim_{j\to\infty}\lambda_j=\lim_{j\to\infty}\gamma_j=0\ \ \ 
\mbox{and}\ \ \ 
B_{\rm EXP-SPT}:=\limsup_{\varepsilon\to0}
\frac{d(\e)\,\log\,j(\e)}{\log\,\log\,\e^{-1}}<\infty,
$$
where 
\begin{eqnarray*}
d(\e)&=&\max\{\,d\in\NN\,:\ \gamma_d\,>\,\e^2\,\},\\
j(\e)&=&\max\{\,j\in\NN\,:\ \lambda_j\,>\,\e^2\,\}.
\end{eqnarray*}
Furthermore, the exponent of EXP-SPT, defined as the
infimum of those $p$ for which the estimate on the information complexity for EXP-SPT holds, 
is equal to $B_{\rm EXP-SPT}$.
\end{itemize}

As we now see, the relaxation from EXP-SPT to EXP-PT is not
essential.
For the un-weighted case, $\gamma_j=1$, both EXP-SPT and EXP-PT
do not hold.

Let us check for which $\lambda_j$'s and $\gamma_j$'s
we have EXP-SPT (and EXP-PT) for the weighted case. 
Since $d(\e)$ goes to infinity as $\e$ tends to
zero, we see 
that $\log\,j(\e)$ must go to infinity slower than $\log\,\log\,\e^{-1}$.
It is easy to check that for $\lambda_j=\exp(-\alpha\,j)$ for some
(maybe very large) $\alpha>0$, we get 
$B_{\rm  EXP-SPT}=\infty$. Consider thus
$\lambda_j=\exp(-\exp(\alpha\,j))$, this time with (maybe very small)
$\alpha>0$. Again, it is easy to check that now
$\log\,j(\e)/(\log(\log(\e^{-1})))$ goes to zero. Obviously we cannot
yet claim EXP-SPT since it also depends on $\gamma_j$'s. By the same
token we conclude that $d(\e)$ must go to infinity slower  than 
$\log\,\log\,\e^{-1}$. So $\gamma_j=\exp(-\exp(\alpha\,j))$ is not
enough. For positive $\alpha_1,\alpha_2$, consider then 
$$
\lambda_j=\exp(-\exp(\alpha_1\,j)) \ \ \ \mbox{and}\ \ \ 
\gamma_j=\exp(-\exp(\exp(\alpha_2\,j))).
$$
We now have 
$$
j(\e)=\frac{1+o(1)}{\alpha_1}\,\log\,\log\,\e^{-1}\ \ \ \mbox{and}\
\ \
d(\e)=\frac{1+o(1)}{\alpha_2}\,\log\,\log\,\log\,\e^{-1}.
$$
Then $B_{\rm EXP-SPT}=0$, so that EXP-SPT now holds with the zero
exponent.  Further examples 
of different notions of exponential tractability will be presented 
in Section \ref{Results}.

\vskip 1pc
Of course, we may say that EXP-SPT (or EXP-PT) is a too strong notion
of exponential tractability.  
We now present a result for a much weaker notion, namely for exponential
weak tractability (EXP-WT) 
which holds if the logarithm of the 
information complexity divided by $d+\log\,\e^{-1}$ goes to zero 
if $\max(d,\e^{-1})$ goes to infinity. We prove that  
EXP-WT holds if and only if
$$
\lim_{j\to\infty}\gamma_j=0 \ \ \ \mbox{and}\ \ \ 
\lim_{j\to\infty}\frac{\log\,\lambda_j^{-1}}{\log\,j}=\infty.
$$
The conditions on the $\lambda_j$'s and the $\gamma_j$'s are now much more lenient
but still do not hold for the un-weighted case. 
{}For the weighted case, we obtain EXP-WT if for $\beta>1$ we have  
$$
\lambda_j=\mathcal{O}\left(\exp\left(-(\log\,j)^\beta\right)\right)
$$
and $\gamma_j$ goes to zero arbitrarily slowly. 
\vskip 1pc
We also consider other notions of exponential tractability
such as EXP-QPT, exponential quasi-polynomial tractability, and
EXP-$(s,t)$-WT, exponential $(s,t)$-weak tractability for 
$\max(s,t)\ge1$.
The corresponding necessary and sufficient conditions on these notions
of exponential tractability are presented in Theorem \ref{thm1}. 
The case of exponential $(s,t)$-weak tractability with $\max(s,t)<1$
as well as EXP-UWT, exponential uniform weak tractability, are left
for future research. 

\vskip 1pc
We end the introduction by presenting a couple of other open problems.
\begin{itemize}
\item In this paper, we assume the class $\Lambda_d^{\rm all}$ 
of all continuous linear functionals as
  information evaluations. It is of a practical interest to
  consider the class $\Lambda^{\rm std}_d$ of only function values. 
In this case we assume that $H_d$ is a reproducing kernel Hilbert
  space so that function values are continuous linear functionals.
The open problem is to find necessary and sufficient conditions 
for various notions of exponential tractability for the class
  $\Lambda^{\rm std}_d$, and to compare them to necessary and sufficient
conditions for the class 
$\Lambda_d^{\rm all}$.
We believe that the recent paper 
\cite{KU2019} may be very helpful for the solution of this problem.
\item As we already mentioned, we consider in this paper only
product weights. It would be of interest to study more general weights 
and to see how the conditions for product weights can be changed. 
\end{itemize}

We summarize the contents of the rest of this paper. 
In Section 2 we  define the problem we study here.
In Section 3 we present the results, and in Section 4 the proofs.

\section{Problem Setting}

We outline the formal setting considered in this paper. 
Let $H_1$ be a separable infinite-dimensional 
Hilbert space 
with inner product 
denoted by $\langle \cdot , \cdot \rangle_{H_1}$. 
Let $G_1$ be an arbitrary Hilbert space, and let $S_1:H_1 \rightarrow G_1$
be a compact linear operator. We stress that compactness of $S_1$ is a
necessary condition to get a finite information complexity and any type
of algebraic or exponential tractability. Then 
$$
W_1=S_1^*S_1\,:\ H_1\to H_1
$$
is also a compact and self-adjoint non-negative operator.  Let
$(\lambda_j,e_j)_{j\in\NN}$ denote its $j^{{\rm th}}$ eigenpair,
$$
W_1 e_j=\lambda_je_j,
$$
where the $e_j$'s are orthonormal and the $\lambda_j$'s non-increasing.
Without loss of generality we assume that $\lambda_1=1$ and,
to omit the trivial problem, that $\lambda_2>0$.  Due to compactness
of $S_1$ we have $\lim_{j\to\infty}\lambda_j=0$. 

Let $d \in \NN$. Define 
$$
H_d:=\underbrace{H_1 \otimes H_1 
\otimes \cdots \otimes H_1}_{d\mbox{{\scriptsize \ times}}}
$$ 
to be the $d$-fold tensor product of $H_1$. 
The inner product is denoted by 
$\langle \cdot ,\cdot \rangle_{H_d}$. Similarly, define
$$
G_d:=\underbrace{G_1 \otimes G_1 
\otimes \cdots \otimes G_1}_{d\mbox{{\scriptsize \ times}}}
$$ 
and the $d$-fold tensor product operator 
$$
S_d:=\underbrace{S_1 \otimes S_1 
\otimes \cdots \otimes S_1}_{d\mbox{{\scriptsize \ times}}}:H_d \rightarrow G_d.
$$
Obviously, $S_d$ is compact. Then
$$
W_d=S_d^*S_d\,:\ H_d\to H_d
$$
is also a compact and
self-adjoint non-negative operator. 
Let $\bsj=(j_1,j_2,\dots,j_d)\in\NN^d$.  
The eigenpairs of $W_d$ are $(\lambda_{d,\bsj},e_{d,\bsj})_{\bsj\in\NN^d}$
with 
$$
\lambda_{d,\bsj}=\prod_{k=1}^d\lambda_{j_k}\ \ \ \mbox{and}\ \ \
e_{d,\bsj}=e_{j_1}\otimes e_{j_2}\otimes\cdots\otimes e_{j_d}.
$$
Hence, we have at least $2^d$ positive eigenvalues of $W_d$.
The square roots of the $\lambda_{d,\bsj}$ are the singular values of $S_d$.

In the following we write $[d]:=\{1,2,\ldots,d\}$. 
Subsets of $[d]$ will be denoted by $\uu$. 
{}From \cite[Sec.~5.3.1]{NW08},    
elements $f \in H_d$ can be decomposed as a sum of mutually
orthogonal elements $f_{\uu}$, $\uu \subseteq [d]$, each of which belongs to
$\bigotimes_{j \in \uu}H_1$, in the form $$f=\sum_{\uu \subseteq [d]} f_{\uu}.$$ 
Furthermore, for $f,g \in H_d$ the inner product is 
$$\langle f , g \rangle_{H_d}= 
\sum_{\uu \subseteq [d]} \langle f_{\uu} , g_{\uu} \rangle_{H_d}$$ and
\begin{equation}\label{eq:norm}
\|f\|_{H_d}^2 =\sum_{\uu \subseteq [d]} \|f_{\uu}\|_{H_d}^2.
\end{equation}
Further information on this orthogonal 
decomposition can be found in \cite[Sec.~5.3.1]{NW08}.   

Eq.~\eqref{eq:norm} shows that 
the contribution of each $f_{\uu}$ is the same, 
which suggests that any group of $f_{\uu}$'s is equally 
important in their contribution to the norm of $f$. 
However, in this paper we are interested in 
the weighted setting which is motivated by the assumption 
that some groups of $f_{\uu}$'s are more important than others, 
or that an element $f$ 
does not depend on some groups of variables at all. 
Such a behavior can be modeled with the help of so-called weights. 
Here we restrict ourselves to {\it product weights} as in the first
paper on weighted spaces \cite{SW98}.
 
Let $\bsgamma=\{\gamma_j\}_{j \in \NN}$ 
be a sequence of non-increasing positive reals, which are called
product weights. The case when some product weights $\gamma_j$'s
are zero is considered in Remark 4. For simplicity we assume that
$\gamma_j\in(0,1]$.  For $\uu \subseteq [d]$ put 
$$
\gamma_{\uu}=\prod_{j\in \uu} \gamma_j,
$$ 
where the empty product is considered to be one, i.e.,
$\gamma_{\emptyset}=1$.
 
Now we define the weighted Hilbert space 
$H_{d,\bsgamma}$ as a separable Hilbert space 
that is algebraically the same as the space $H_d$ 
but whose inner product for $f,g \in H_d$ is given by 
\begin{equation}\label{nonzeroweights}
\langle f,g\rangle_{H_{d,\bsgamma}}:=
\sum_{\uu \subseteq [d]} \gamma_{\uu}^{-1} \langle f_{\uu},g_{\uu}\rangle_{H_d}.
\end{equation}
{}For product weights, the space $H_{d,\bsgamma}$ 
can also be described as a tensor product space, namely, 
$$
H_{1,\gamma_1}\otimes H_{1,\gamma_2}
\otimes \cdots \otimes H_{1,\gamma_d}.
$$ 
In particular, if $\gamma_j=1$ for all $j \in \NN$, 
then we have $H_{d,\bsgamma}=H_d$.

We study the sequence $\cS_{\bsgamma}=\{S_{d,\bsgamma}\}_{d \in \NN}$ 
of operators given by 
$$
S_{d,\bsgamma}:H_{d,\bsgamma}\rightarrow G_d, \ \ \ S_{d,\bsgamma}
(f)=S_d (f)
$$ 
and consider the problem of approximating 
$S_{d,\bsgamma}(f)$ in the norm of $G_d$ 
for elements $f$ from the unit ball of
$H_{d,\bsgamma}$. We shall show in the next section
how the weights affect the singular values of $S_{d,\bsgamma}$. 

The elements $S_{d,\bsgamma}(f)$ are approximated by 
algorithms $A_{d,n}(f)$ which use at most 
$n$ information evaluations from the class $\Lambda_d^{\rm all}=
H_{d,\bsgamma}^*$ which consists of all continuous linear functionals
defined on $H_{d,\bsgamma}$.
The general form of an algorithm $A_{d,n}$ is
$$
A_{d,n}(f)=\phi_{d,n}(L_1(f),L_2(f),\dots,L_d(f)),
$$ 
where $L_j\in\Lambda_d^{\rm all}$ and $\phi_{d,n}:\RR^n \to G_d$ is any
mapping. The choice of $L_j$ can be adaptive, i.e., it may
depend on the previously computed information
$L_1(f),L_2(f),\dots,L_{j-1}(f)$.

The error is studied in the worst-case setting and is 
defined as $$e(A_{d,n})=\sup_{f \in H_{d,\bsgamma} \atop
  \|f\|_{H_{d,\bsgamma}} \le
  1}\|S_{d,\bsgamma}(f)-A_{d,n}(f)\|_{G_d}.
$$ 
Denote by $e_0$ the initial error, i.e., 
$$
e_0=\sup_{f \in H_{d,\bsgamma} 
\atop \|f\|_{H_{d,\bsgamma}} \le 1}\|S_{d,\bsgamma}(f)\|_{G_d},
$$ 
which is just the operator norm of $S_{d,\bsgamma}$.

We are interested in the minimal number $n$ of
information evaluations from the class $\Lambda_d^{\rm all}$ 
in order to reduce the initial error by a factor of $\varepsilon \in (0,1]$. 
To this end let 
$$
e(n,S_{d,\bsgamma})= \inf_{A_{d,n}}e (A_{d,n})
$$ 
be the $n^{{\rm th}}$ minimal error, 
where the infimum is extended over all algorithms 
$A_{d,n}$ which use at most $n$ information evaluations  
from the class $\Lambda_d^{\rm all}$. 
Then we study the information complexity 
for the normalized error criterion, which is defined by 
$$
n(\varepsilon,S_{d,\bsgamma})=
\min\{n \, : \, e(n,S_{d,\bsgamma}) \le \varepsilon\, e_0\}.
$$
It is known that for the class $\Lambda_d^{\rm all}$ the information
complexity is fully characterized in terms of the singular values of 
$S_{d,\bsgamma}$, or equivalently, in terms of the eigenvalues of
$W_{d,\bsgamma}=S_{d,\bsgamma}^*S_{d,\bsgamma}$.

We are interested in the behavior of 
$n(\varepsilon,S_{d,\bsgamma})$ when $\max(d,\e^{-1})$ goes to
infinity in an arbitrary way.
This is the subject of tractability, see~\cite{NW08,NW10,NW12}. 
The notions of tractability classify the order of growth of 
the information complexity.
Standard tractability is studied 
in the {\it algebraic setting} (ALG), as it is called nowadays. 
In this case, one describes the dependence of $n(\varepsilon,S_{d,\bsgamma})$ 
on the dimension and the error threshold, i.e.,  with respect to the pair 
$(d,\varepsilon)$.
Recently also the {\it exponential setting} (EXP) gained much attention, 
and this setting is the central topic of the present paper. 
In the exponential setting, we study 
how $n(\varepsilon,S_{d,\bsgamma})$ behaves with respect to the 
pair $(d,1+\log\varepsilon^{-1})$. 
We are ready to define various notions of EXP tractabilities. 

The problem $S_{\bsgamma}=\{S_{d,\bsgamma}\}$ is said to be: 

\begin{itemize}
 \item {\it Exponentially strongly polynomially 
tractable (EXP-SPT)} if there are $C, p \ge 0$ such that 
 \begin{equation}\label{def:ESPT}
 n(\varepsilon,S_{d,\bsgamma}) \le 
C (1 + \log  \varepsilon^{-1})^p \ \ \ \ 
\forall d \in \NN,\ \forall \varepsilon \in (0,1].
 \end{equation}
 The infimum over all exponents $p \ge0$ such that 
\eqref{def:ESPT} holds for some $C \ge 0$ is called 
the exponent of EXP-SPT and is denoted by $p^*$.
 \item {\it Exponentially polynomially tractable (EXP-PT)} if 
there are $C, p, q \ge 0$ such that 
$$
n(\varepsilon,S_{d,\bsgamma}) \le C d^{\,q} (1 + \log\,
  \varepsilon^{-1})^p \ \ \ \ 
\forall d \in \NN,\ \forall \varepsilon \in (0,1].$$ 
 \item {\it Exponentially quasi-polynomially tractable 
(EXP-QPT)} if there are $C, t \ge 0$ such that 
 \begin{equation}\label{def:EQPT}
 n(\varepsilon,S_{d,\bsgamma}) 
\le C \exp(t(1+\log\,d)(1+\log(1 + \log\,
\varepsilon^{-1})))\ \ \ \ \forall d \in \NN,\ \forall \varepsilon \in (0,1].
 \end{equation}
 The infimum over all exponents $t \ge0$ such that 
\eqref{def:EQPT} holds for some 
$C \ge 0$ is called the exponent of EXP-QPT and is denoted by $t^*$.
 \item {\it Exponentially $(s,t)$-weakly tractable (EXP-$(s,t)$-WT)} 
for positive $s$ and $t$ if
$$
\lim_{d+\varepsilon^{-1} \rightarrow \infty} 
\frac{\log \max(1, n(\varepsilon, 
S_{d,\bsgamma}))}{d^t + (1 + \log  \varepsilon^{-1})^s} = 0.
$$ If $s=t=1$, we speak of 
{\it exponential weak tractability (EXP-WT)}.
\item {\it Exponentially uniformly weakly tractable (EXP-UWT)}
if EXP-$(s,t)$-WT holds for all positive $s$ and $t$.

\end{itemize}

To shorten the notation, we often say that the problem $S_{\bsgamma}$
is EXP-SPT, EXP-PT, etc., by saying that EXP-SPT, EXP-PT, etc., 
holds. As already mentioned, we do not consider EXP-$(s,t)$-WT with
$\max(s,t)<1$ and exponential uniform tractability in this paper.

\begin{remark}\rm
In some papers, for example in \cite{PPW17}, the notion of
EXP-$(s,t)$-WT is called $(s,\ln^{\kappa})$-weak tractability, where
$s$ corresponds to $t$ and $\kappa$ corresponds to $s$ in our notation.

The notions of quasi-polynomial, $(s,t)$-weak and
uniform weak tractabilities in the algebraic case were for the first
time defined correspondingly in
\cite{GW11,S13,SW15}. Here we adopt these concepts for exponential
tractability by replacing $\e^{-1}$ by $1+\log\,\e^{-1}$.   
\end{remark}
\vskip 1pc

The main result of this work, a characterization of 
weighted linear tensor product problems with respect 
to exponential tractability, will be stated in the next 
section as Theorem~\ref{thm1}. 
The proofs will be presented in Section~\ref{sec:proof}.

\section{The results}\label{Results}

To begin with we introduce another, for our purpose more convenient, 
representation of the information complexity.  It is known 
from \cite{TWW}, see also \cite{NW08}, 
how the information complexity depends on the 
singular values of $S_{d,\bsgamma}$, which 
are the same as the square-roots of the eigenvalues 
of the compact self-adjoint and positive 
definite linear operator 
$W_{d,\bsgamma} = S_{d,\bsgamma}^* S_{d,\bsgamma} : 
H_{d,\bsgamma}\rightarrow H_{d,\bsgamma}$.

Let $\{\lambda_j\}_{j \in \NN}$ and 
$\{\gamma_j\}_{j \in \NN}$ 
be as in the previous section. For $\bsj=(j_1,\ldots,j_d)\in \NN^d$ 
define 
$$
\uu(\bsj):=\{k \in [d]\ : \ j_k \ge 2\} \ \ \ 
\mbox{ and }\ \ \ \lambda_{d,\bsj}:=
\lambda_{j_1}\lambda_{j_2}\cdots \lambda_{j_d}.
$$ 
{}From \cite[Section~5.3]{NW08}) we know that the eigenvalues of 
$W_{d,\bsgamma}$ are  
$$
\lambda_{d,\bsgamma,\bsj}:=
\gamma_{\uu(\bsj)} \lambda_{d,\bsj}=
\left(\prod_{k=1 \atop j_k \ge 2}^d \gamma_k\right) 
\lambda_{j_1}\cdots \lambda_{j_d} = \left(\prod_{k\in \uu(\bsj)} \gamma_k \right) \lambda_{d,\bsj}.
$$ 
Clearly, $\lambda_{d,\bsgamma,\bsj}$ is maximized for
$\bsj=(1,1\dots,1)$ and then it is equal to $1$. 
Hence, $\|W_{d,\bsgamma}\|=\|S_{d,\bsgamma}\|=1$ and 
the initial error $e_0$ is also one. This means that the problem is well
normalized for all $d\in\NN$ and all product weights $\bsgamma$.

The information complexity is now

\begin{equation}\label{eq:infocomp_alt}
n(\varepsilon,S_{d,\bsgamma})=
 \left|\left\{\bsj \in \NN^d\ : \ \lambda_{d,\bsgamma,\bsj} > 
 \varepsilon^2 \right\}\right|.
\end{equation}
Define $$\lambda_{k,j}=\left\{
\begin{array}{ll}
1 & \mbox{ if } j=1,\\
\gamma_k \lambda_j & \mbox{ if } j \ge 2. 
\end{array}\right.$$ 
Then $$
\lambda_{d,\bsgamma,\bsj}=\prod_{k=1}^d \lambda_{k,j_k}$$ and hence 
\begin{equation}\label{def:infocomp}
n(\varepsilon,S_{d,\bsgamma})=|A_{\varepsilon,d}|,\ \ \ 
\mbox{where}\ A_{\varepsilon,d}=\{(n_1,\ldots,n_d)\in \NN^d \ : \ \lambda_{1,n_1}\cdots \lambda_{d,n_d}> \varepsilon^2\}.
\end{equation}
Clearly, 
$$
n(\e,S_{d,\bsgamma})\le n(\e_1,S_{d_1,\bsgamma})\ \ \
\mbox{for all $\e_1\le \e$ and $d_1\ge d$}.
$$
Hence, for decreasing $\e$ and increasing $d$, the information complexity
is non-increasing.

In the sequel we will 
work with the representation of the information complexity in~\eqref{def:infocomp}. 
We show how weighted tensor product problems can be classified 
with respect to different notions of EXP tractability by means 
of the eigenvalues $\{\lambda_j\}_{j \in \NN}$ of the operator $W_1:=S^\ast
S: H_1 \rightarrow H_1$ and
of the weights $\{\gamma_j\}_{j \in \NN}$.
We remind the reader of what we assume about the $\lambda_j$'s and
$\gamma_j$'s. We have
\begin{equation}\label{def:ev}
1=\lambda_1 \ge \lambda_2 \ge \lambda_3 \ge \ldots \ge 0,\ \
\mbox{with $\lambda_2>0$}.
\end{equation}
Note that for $\lambda_2=0$ the problem becomes trivial
since $n(\e,S_{d,\bsgamma})=1$ for all $\e\in[0,1)$ and $d\in\NN$.
On the other hand, if $\lambda_1\not=1$ then the problem is not well 
normalized. In this case, the initial error $e_0$ is $\lambda_1^d$ and 
for  the normalized 
error criterion we may work with $\lambda_{k,n_j}/\lambda_1$ instead of $\lambda_{k,n_j}$. By assuming
that $\lambda_1=1$ we simplify the notation. 
Note also that $\lim_{j\to\infty}\lambda_j=0$ iff $S_1$ (as well as
$S_{d,\bsgamma}$) is  compact. This assumption implies that the
information complexity $n(\e,S_{d,\bsgamma})$ is finite for all $\e>0$
and all $d\in \NN$.   
 
For $\varepsilon \in (0,1)$ and $\lim_{j\to\infty}\lambda_j=0$, define 
\begin{equation}
 j(\varepsilon) =  \max\{j \in \NN \ : \ \lambda_j > \varepsilon^2\}.
\end{equation}
Then $j(\e)$ is well defined and always finite. Since $\lambda_1=1$, 
we have $j(\e)\ge1$.  
Furthermore, $j(\e)$  goes to
infinity if and only if all $\lambda_j$'s are positive.  
\vskip 1pc
We also assume that the weights satisfy 
\begin{equation}\label{posweights}
1\ge \gamma_1 \ge \gamma_2 \ge \gamma_3 \ge \ldots > 0.
\end{equation}
The ordering of the $\gamma_j$'s tells us that the successive subproblems 
are less and less important. The assumption that the weights are at
most one is made for simplicity to guarantee that $\lambda_{k,j}\le
\lambda_{k,1}=1$. 
The case of more general $\lambda_j$'s and
$\gamma_j$'s  is considered for algebraic tractability in \cite[Section~5.3]{NW08}.

For $\varepsilon\in(0,1)$ and
$\lim_{j\to\infty}\gamma_j=0$, define
$$
d(\varepsilon) =  \max\{d \in \NN \ : \ \gamma_d\,>\,\varepsilon^2\}.
$$
Then $d(\e)$ is well defined. We put 
$d(\e)=0$ for $\gamma_1\le\e^2$, and note that $d(\e)\ge1$ for $\gamma_1>\e^2$.
Both~$j(\e)$ and $d(\e)$ 
are non-decreasing, and 
$\lim_{\varepsilon\to0}d(\varepsilon)=\infty$. 

Now we are able to state our main result.
To shorten the notation we write ''iff'' instead of ``if and only
if''.

\begin{theorem}\label{thm1}
We have
\begin{enumerate}
\item EXP-SPT holds iff 
$$
\lim_{j\to\infty}\lambda_j=\lim_{j\to\infty}\gamma_j=0\ \ \ 
\mbox{and}\ \ \ 
B_{{\rm EXP-SPT}}:=\limsup_{\varepsilon \rightarrow 0}
\frac{d(\varepsilon) \log j(\varepsilon)}{\log\log
  \frac{1}{\varepsilon}}< \infty.$$  
If this holds then the exponent of EXP-SPT is $p^*=B_{{\rm EXP-SPT}}$.
\item EXP-SPT and EXP-PT are equivalent.
\item EXP-QPT holds iff 
$$
\lim_{j\to\infty}\lambda_j=\lim_{j\to\infty}\gamma_j=0\ \ \ 
\mbox{and}\ \ \ 
B_{{\rm EXP-QPT}}:=\limsup_{\varepsilon \rightarrow 0} 
\frac{d(\varepsilon) \log j(\varepsilon)}{[\log d(\varepsilon)]
\log\log \frac{1}{\varepsilon}}< \infty.
$$
If this holds then the exponent of EXP-QPT is $t^*=B_{{\rm EXP-QPT}}$.

\item  Let $s=t=1$. 

EXP-WT holds iff 
$$\lim_{j \rightarrow \infty} \gamma_j=0\ \ 
\mbox{ and }\ \ 
\lim_{j \rightarrow \infty} 
\frac{\log \frac{1}{\lambda_j}}{\log j}=\infty.
$$
\item  Let $s=1$  and $t<1$.

EXP-$(1,t)$-WT holds iff 
$$
\lim_{j \rightarrow \infty} 
\frac{\log \frac{1}{\gamma_j}}{\log j}=\infty\ 
\ \mbox{ and }\ \ \lim_{j \rightarrow \infty} 
\frac{\log \frac{1}{\lambda_j}}{\log j}=\infty.
$$
\item Let $s=1$ and $t>1$.

EXP-$(s,t)$-WT holds iff 
$$
\mbox{$\gamma_j$'s are arbitrary}\ \ \ \mbox{and}\ \ \ 
 \lim_{j \rightarrow \infty} \frac{\log \frac{1}{\lambda_j}}{\log
   j}=\infty.
$$
\item Let $s>1$, $t\le1$ and $\lambda_2<1$. 

EXP-$(s,t)$-WT holds iff 
$$
\mbox{$\gamma_j$'s are arbitrary}\ \ \ \mbox{and}\ \ \ 
\lim_{j \rightarrow \infty}\frac{(\log\frac{1}{\lambda_j})^s}{\log
  j}=\infty.
$$
\item
Let $s>1$, $t\le1$ and $\lambda_2=1$. 

EXP-$(s,t)$-WT holds 
iff 
$$
\exists\,p\in\NN \ \ \mbox{with}\ \
\gamma_p<1\ \ \ \mbox{and}\ \ \ 
\lim_{j \rightarrow \infty}\frac{(\log\frac{1}{\lambda_j})^s}{\log
  j}=\infty.
$$
\item  Let $s>1$ and $t>1$.

EXP-$(s,t)$-WT holds 
iff
$$
\mbox{$\gamma_j$'s are arbitrary}\ \ \ \mbox{and}\ \ \ 
\lim_{j \rightarrow \infty}\frac{(\log\frac{1}{\lambda_j})^s}{\log
  j}=\infty.
$$
\item
Let $s<1$ and $t>1$. 
EXP-$(s,t)$-WT holds for arbitrary $\gamma_j$'s iff 
$$
\lim_{j\to\infty}\frac{\left(\log \frac1{\lambda_j}\right)^{\eta}}
{\log j}=\infty \ \ \mbox{with}\ \
\eta=\frac{s(t-1)}{t-s}.
$$
\item
Let $s<1$ and $t=1$.

EXP-$(s,1)$-WT holds iff for arbitrary integers $d,k,j$ with $j\ge 2$ and $k\le d$ it is true that
\begin{equation}\label{111}
\lim_{d+\gamma_k^{-d}\lambda_j^{-d}\to\infty}
\frac{\left(\log\frac1{\gamma_k}\right)^s+\left(\log\frac1\lambda_j\right)^s
}{d^{1-s}\log j}=\infty.
\end{equation}
\end{enumerate}
\end{theorem}

Before we present the proof of Theorem \ref{thm1} we illustrate some of the results and discuss their meaning.
 
\begin{example}\rm
Let $\lambda_j=\exp(-\exp(j^{\alpha}))$ and $\gamma_j=\exp(-\exp(j^{\beta}))$ for positive $\alpha$ and  $\beta$. Then we have $$j(\varepsilon)=\left\lceil \left(\log\log \frac{1}{\varepsilon^2}\right)^{1/\alpha}\right\rceil -1 \ \ \mbox{ and }\ \ d(\varepsilon)=\left\lceil \left(\log\log \frac{1}{\varepsilon^2}\right)^{1/\beta}\right\rceil -1$$ and hence
\begin{itemize}
 \item $\alpha >0$ and $\beta>1$ imply EXP-SPT with $p^*=0$;
 \item $\alpha >0$ and $\beta=1$ imply EXP-QPT with $t^*=1$, 
but EXP-SPT does not hold.
\end{itemize}
\end{example}

Note that Items {\it 4.}-{\it 11.} of Theorem~\ref{thm1} give a full characterization of EXP-$(s,t)$-WT for all $(s,t) \in (0,\infty)^2 \setminus (0,1)^2$. The following remarks are in order.

\begin{remark}\rm
The condition $$ \lim_{j \rightarrow \infty} \frac{\log \frac{1}{\lambda_j}}{\log j}=\infty$$ is satisfied if and only if $\lambda_j$ is of the form $$\lambda_j=\frac{1}{j^{h(j)}}$$ where $h:\NN \rightarrow \RR$ satisfies $\lim_{j \rightarrow \infty} h(j)=\infty$.  So, for example, we have EXP-WT if $\lim_{j \rightarrow \infty}\gamma_j=0$ and $\lambda_1=1$, $\lambda_2=0.95$ and $\lambda_j=j^{-\log\log j}$ for $j \ge 3$.
\end{remark}

\begin{remark}\rm
Consider $s>1$, $t\le1$ and $\lambda_2\le1$ described in Items {\it 7.} and {\it 8.} 
of Theorem~\ref{thm1}:

For $\lambda_2<1$, we have a single largest eigenvalue and
EXP-$(s,t)$-WT holds for arbitrary $\gamma_j$'s as long as
$\lim_{j \rightarrow \infty}(\log \lambda_j^{-1})^s/\log j=\infty$. In particular, this holds
for the un-weighted case,
$\gamma_j=1$ for all $j \in \NN$.
 
For $\lambda_2=1$, we have a multiple largest eigenvalue and
EXP-$(s,t)$-WT holds under the same conditions on the $\lambda_j$'s but
now we need to assume that not all $\gamma_j$'s are one.
In particular, this holds for
$$
1=\gamma_1=\ldots=\gamma_{p-1}>\gamma_p=\gamma_{p+1}=\ldots>0.
$$  
\end{remark}
\begin{remark}\rm
Consider $s\ge1$ and $t>1$ described in Items {\it 6.} and {\it 9.} 
of Theorem \ref{thm1}. Then EXP-$(s,t)$-WT holds for arbitrary 
$\gamma_j$'s, i.e., even for the un-weighted case $\gamma_j=1$,
and for $\lambda_j$ satisfying the same condition as before.
This case was proved in \cite{HKW19}.
\end{remark}

\begin{remark}\rm 
We briefly note what happens if some weights in
\eqref{posweights} are zero, say $\gamma_{j^*}=0$ for some $j^*\in\NN$.
Obviously, monotonicity of the $\gamma_j$'s implies that $\gamma_j=0$ 
for all $j\ge j^*$. 
Then $\gamma_{\uu}=0$ for all $\uu$ containing  one or more indices
 at least equal to $j^*$.  For such $\uu$, we must assume in
\eqref{nonzeroweights} that $f_{\uu}=0$ and adopt 
the convention that $0/0=0$. In this case, $H_{d,\bsgamma}$ is
algebraically a proper subset of $H_d$. 

Assume first that $j^*=1$. Then the only non-zero eigenvalues are
$\lambda_{k,1}=1$. This means that the problem is trivial since 
$n(\e,S_{d,\bsgamma})=1$ for all $\e\in[0,1)$ and $d\in \NN$.

Assume then that $j^*\ge 2$. It is easy to check that we now have
$$
n(\e,S_{d,\bsgamma})\le n(\e,S_{j^*-1,\bsgamma})
\ \ \ \mbox{for all $\e\in (0,1)$ and $d\in\NN$.}
$$ 
Hence, $d(\e)\le j^*-1$ and $d(\e)=j^*-1$ for
$\e<\gamma_{j^*-1}^{1/2}$. The factors $d(\e)$
and $d(\e)/\log\,d(\e)$ cannot change the fact that $B_{\rm EXP-SPT}$ or
$B_{\rm EXP-QPT}$ are finite, and then
EXP-SPT, EXP-PT and EXP-QPT are equivalent.
\end{remark}

\section{The proofs}\label{sec:proof}

We first show how the information complexity can be bounded in terms
of $j(\varepsilon)$ and $d(\varepsilon)$. 

\begin{lemma}\label{le1}
If $\lim_{j\to\infty}\lambda_j=\lim_{j\to\infty}\gamma_j=0$ then
for $\varepsilon\in (0,1)$ we have 
$$
n(\varepsilon,S_{d,\bsgamma}) \le j(\varepsilon)^{\min(d,d(\varepsilon))} 
\le n(\varepsilon^{2 d(\varepsilon)},S_{d(\varepsilon),\bsgamma}),
$$
and for $d\ge d(\varepsilon)$ 
$$
n(\varepsilon,S_{d,\bsgamma})=n(\varepsilon,S_{d(\varepsilon),\bsgamma}).
$$
\end{lemma}

\begin{proof}
We use \eqref{def:infocomp}. Consider the eigenvalue $\lambda_{1,n_1} \lambda_{2,n_2}\cdots \lambda_{d,n_d}$. 
\begin{itemize}
\item If $n_k \ge j(\varepsilon)+1$ (in particular $n_k \ge 2$) for some $k \le d$, then we have $$ \lambda_{1,n_1} \lambda_{2,n_2}\cdots \lambda_{d,n_d} \le \lambda_{k,n_k}=\gamma_k \lambda_{n_k} \le \lambda_{n_k} \le \varepsilon^2$$ so that $(n_1,\ldots,n_d)\not\in A_{\varepsilon,d}$.
\item If $d \ge d(\varepsilon)+1$ and $n_d \ge 2$, then $$ \lambda_{1,n_1} \lambda_{2,n_2}\cdots \lambda_{d,n_d} \le  \gamma_d \lambda_2 \le \gamma_d \le \varepsilon^2$$ and again $(n_1,\ldots,n_d)\not\in A_{\varepsilon,d}$.
\end{itemize}
Hence, only 
$$
(n_1,n_2,\dots,n_{\min(d,d(\varepsilon))},1,1,\dots,1)\in \NN^d
$$
for $n_j\in\{1,2,\dots,j(\e)\}$ with $j=1,2,\dots,\min(d,d(\varepsilon))$  
may belong to $A_{\varepsilon,d}$, and therefore  
$$
n(\varepsilon,S_{d,\bsgamma})=|A_{\varepsilon,d}| \le
j(\varepsilon)^{\min(d,d(\varepsilon))}.
$$ 
Furthermore for $d\ge d(\e)$, we have 
$$
n(\e,S_{d,\bsgamma})=|A_{\e,d(\e)}|=n(\e,S_{d(\e),\bsgamma}),
$$
as claimed.

In order to show the remaining  inequality we consider the eigenvalues 
$$
\mbox{$\lambda_{1,n_1} \lambda_{2,n_2}\cdots 
\lambda_{d(\varepsilon),n_{d(\varepsilon)}}$ 
for $n_j\in \{1,\ldots,j(\varepsilon)\}$.}
$$ 
For these eigenvalues we have $$\lambda_{1,n_1}\lambda_{2,n_2}\cdots
\lambda_{d(\varepsilon),n_{d(\varepsilon)}} \ge
(\gamma_{d(\varepsilon)} \lambda_{j(\varepsilon)})^{d(\varepsilon)} >
\varepsilon^{4d(\varepsilon)}.$$ This implies that we have at least
$j(\varepsilon)^{d(\varepsilon)}$ eigenvalues no less than
$\varepsilon^{4 d(\varepsilon)}$. 
Hence 
$$
j(\varepsilon)^{\min(d,d(\varepsilon))}\le j(\varepsilon)^{d(\varepsilon)} 
\le n(\varepsilon^{2 d(\varepsilon)},S_{d(\varepsilon),\bsgamma}).
$$
This completes the proof.
\end{proof}
\vskip 1pc
The next technical lemma will help to state the conditions for various 
notions of exponential tractability in a concise form.

\begin{lemma}\label{le2}
Let $\{a_j\}_{j \in \NN}$ be a non-increasing 
sequence of positive reals. Then we have 
\begin{equation}\label{le2eq1}
M_c:=\sum_{j=1}^{\infty} a_j^c < \infty \ \ \ \ \mbox{ for all $c>0$}
\end{equation}
if and only if 
\begin{equation}\label{le2eq2}
\lim_{j \rightarrow \infty}\frac{\log \frac{1}{a_j}}{\log j}=\infty.
\end{equation}
\end{lemma}

\begin{proof}
Assume that \eqref{le2eq1} holds. We have $$n a_n^c \le
a_1^c+\cdots+a_n^c \le M_c$$ and hence
$$\frac{1}{a_n} \ge \frac{n^{1/c}}{M_c^{1/c}}.$$ Taking the
logarithm we obtain $$ \log \frac{1}{a_n} 
\ge \frac{1}{c} \log n - \frac{1}{c}\log M_c$$ and therefore  
$$\liminf_{n \rightarrow \infty}\frac{\log \frac{1}{a_n}}{\log n} \ge \frac{1}{c}.$$ Now \eqref{le2eq2} follows by letting $c \rightarrow 0$.

If \eqref{le2eq2} holds then for every $c>0$ there exists a number $j_c>0$ such that $$\frac{\log \frac{1}{a_j}}{\log j} \ge \frac{2}{c}\ \ \ \mbox{for all $j \ge j_c$.}$$ This implies $$a_j^c \le \frac{1}{j^2} \ \ \ \mbox{for all $j \ge j_c$.}$$ Hence \eqref{le2eq1} holds.
\end{proof}

\begin{lemma}\label{le3}
For $s,a_1,a_2,\dots,a_m\ge0$ and $m\in\NN$, we have
$$
(a_1+\dots+a_m)^s=\alpha_{s,m}\left(a_1^s+\dots +a_m^s\right)
\ \ \ \mbox{with} \ \ \ 0^0=1,
$$
where $\alpha_{s,m}$ also depends on $a_1,\dots,a_m$ but is uniformly
bounded in the $a_j$'s,
$$
\alpha_{s,m}\in[1,m^{s-1}]\ \ \mbox{for $s\ge1$}\ \ \ 
\mbox{and}\ \ \ \alpha_{s,m}\in[m^{s-1},1]\ \ \mbox{for $s<1$}. 
 $$
\end{lemma}
\begin{proof}
It is well known that for $s>1$ we have
$$
m^{1-s}(a_1+\dots+a_m)^s\le a_1^s+\dots+a_m^s\le (a_1+\dots+a_m)^s,
$$
whereas for $s<1$ we have
$$
(a_1+\dots+a_m)^s\le a_1^s+\dots+a_m^s\le m^{1-s}(a_1+\dots+a_m)^s.
$$
This can be rewritten as
$$
(a_1+\dots+a_m)^s=\alpha_{s,m}(a_1^s+\dots+a_m^s)
$$
with $\alpha_{s,m}$ satisfying the bounds in Lemma \ref{le3}.
\end{proof}
We need a necessary condition on EXP-$(s,t)$-WT.
\begin{lemma}\label{le4}
For any positive $s,t$ and integers $k_1,k_2,\dots,k_d$ with
$k_j\ge2$, 
EXP-$(s,t)$-WT implies that
$$
\lim_{d+\max_{j \in [d]}k_j\to\infty}\frac{d^t+\left(\sum_{j=1}^d\log
    \frac1{\gamma_j}\right)^s+\left(\sum_{j=1}^d\log\frac1{
\lambda_{k_j}}\right)^s}
{\sum_{j=1}^d\log\,k_j}=\infty. 
$$
\end{lemma}
\begin{proof}
EXP-$(s,t)$-WT implies that
$$
\lim_{d+\e^{-1}\to\infty}
\frac{\log n(\e,S_{d,\bsgamma})}{d^t+\left(\log
\e^{-1}\right)^s}=0.
$$
Take $\e^2=\gamma_1\lambda_{k_1}\cdots\gamma_d\lambda_{k_d}\alpha$
with $\alpha<1$ . We take $k_d$ large enough so that $\e<1$. Then 
$\lambda_{1,j_1}\cdots\lambda_{d,j_d}>\e^2$ for all
$j_{1}=1,2,\dots,k_1,\ j_2=1,2,\dots,k_2$ and $j_d=1,2,\dots,k_d$.
 Hence $n(\e,S_{d,\bsgamma})\ge 
\prod_{j=1}^dk_j$ and
$$
\lim_{d+\e^{-1}\to\infty}\frac{\sum_{j=1}^d\log k_j}{d^t
+\frac{1}{2^s}\left(\sum_{j=1}^d\log\frac1{\gamma_j}\,+\,
\sum_{j=1}^d\log\frac1{\lambda_{k_j}}\,+\,\log\frac1{\alpha}\right)^s}=0.
$$
Applying now Lemma \ref{le3} with $m=3$ we obtain
\begin{eqnarray*}
\lefteqn{\left(\sum_{j=1}^d\log\frac1{\gamma_j}\,+\,
\sum_{j=1}^d\log\frac1{\lambda_{k_j}}\,+\,\log\frac1{\alpha}\right)^s}\\
& =& \left(
\left(\sum_{j=1}^d\log\frac1{\gamma_j}\right)^s\,+\,
\left(\sum_{j=1}^d\log\frac1{\lambda_{k_j}}\right)^s\,+\,\left(
\log\frac1{\alpha}\right)^s\right)\alpha_{s,3}.
\end{eqnarray*}

Note that $d+\e^{-1}\to \infty$ is equivalent to $d+\max_{j\in [d]}k_j\to\infty$.
Taking the reciprocal this yields
$$
\lim_{d+\max_{j\in [d]}k_j\to\infty}
\frac{d^t+\frac{\alpha_{s,3}}{2^s}
\left(\left(\sum_{j=1}^d\log\frac1{\gamma_j}\right)^s
+\left(\sum_{j=1}^d\log \frac1{\lambda_{k_j}}\right)^s
+\left(\log\frac1{\alpha}\right)^s\right)}{\sum_{j=1}^d\log \,k_j}=\infty.
$$
Since $(\log \frac1{\alpha})^s/(\sum_{j=1}^d \log k_j)$ tends to zero, 
and since we may increase the numerator of the last expression by
multiplying $d^t$ by $\max(1,\alpha_{s,3}/2^s)$, we obtain Lemma \ref{le4}.
\end{proof}

We are ready to turn to the proof of the main result of the paper.

\begin{proof}[Proof of Theorem~\ref{thm1}]
$\quad$

\begin{enumerate}
\item Assume first that 
$\lim_{j \rightarrow \infty}\lambda_j=\lim_{j \rightarrow \infty}\gamma_j=0$ and 
$B:=B_{{\rm EXP-SPT}} < \infty$. Let $\delta \in (0,\infty)$. Then
from the definition of $B$ we have that there exists an
$\varepsilon_{\delta}\in(0,\gamma_1^{1/2}/{\rm e})$, with ${\rm e}=\exp(1)$,
such that 
$$\log
j(\varepsilon)^{d(\varepsilon)} \le \log \left( \log
  \frac{1}{\varepsilon}\right)^{B+\delta} \ \ \ \mbox{ for all
  $\varepsilon \in (0,\varepsilon_{\delta}]$.}
$$ 
Hence, according to
Lemma~\ref{le1} we have 
$$
n(\varepsilon,S_{d,\bsgamma}) \le 
\left( \log
  \frac{1}{\varepsilon}\right)^{B+\delta}\le   
\left(1+ \log
  \frac{1}{\varepsilon}\right)^{B+\delta}  
\ \ \ \mbox{ for all
  $\varepsilon \in (0,\varepsilon_{\delta}]$}.
$$ 
Consider now $\e\in[\e_\delta,1]$  and $d\in\NN$. Let
$$
C_\delta=\left(1+\log \frac1{\e_\delta}\right)^{B+\delta}.
$$
Since 
$$
\left(1+\log \frac1{\e_\delta}\right)^{B+\delta}\le C_\delta
\left(1+\log \frac1{\e}\right)^{B+\delta},
$$
we have
$$
n(\e,S_{d,\bsgamma})
\le 
n(\e_\delta,S_{d,\bsgamma})
\le
\left(1+\log \frac1{\e_\delta}\right)^{B+\delta}\le 
C_\delta \left(1+\log \frac1{\e}\right)^{B+\delta}.
$$
Hence,
$$
n(\e,S_{d,\bsgamma})
\le C_\delta \left(1+\log \frac1{\e}\right)^{B+\delta}
\ \ \ \mbox{for all $\e\in(0,1)$ and $d\in\NN$},
$$
 which means EXP-SPT with the exponent $p^*\le B$.
\vskip 1pc
On the other hand, assume that we have EXP-SPT with the exponent $p^*<
\infty$. Hence, for every $\delta>0$ there exists a number $C_{\delta}$
such that for all $\varepsilon \in (0,1)$ and $d\in\NN$  
$$
n(\varepsilon,S_{d,\bsgamma})\le C_\delta\left(1+\log\frac1{\varepsilon}
\right)^{p^*+\delta}.
$$

We first show that $\lim_{j \rightarrow \infty}\lambda_j=\lim_{j \rightarrow \infty}\gamma_j=0$ which is even
known for the algebraic complexity. For completeness we provide a
short proof.
 
The condition $\lim_{j \rightarrow \infty}\lambda_j=0$ easily follows from the compactness of
$S_{d,\bsgamma}$ since otherwise $n(\varepsilon,S_{d,\bsgamma})=\infty$ for
small positive $\varepsilon$. 

The condition $\lim_{j \rightarrow \infty}\gamma_j=0$ is also easy to show since otherwise
due to the monotonicity of the $\gamma_j$'s we have $\lim_{j \rightarrow \infty}\gamma_j=\gamma^*>0$.
Then we can take $2^d$ eigenvalues
$\lambda_{1,n_1}\lambda_{2,n_2}\cdots\lambda_{d,n_d}$ with
$n_j\in\{1,2\}$. Then each such eigenvalue is at least 
$(\gamma^*\lambda_2)^d$ and 
\begin{equation}\label{gammazero}
n(\tfrac12(\gamma^*\lambda_2)^{d/2},S_{d,\bsgamma})\ge 2^d.
\end{equation}
This contradicts EXP-SPT. 
 
We now apply EXP-SPT for $n(\varepsilon^{2 d(\varepsilon)},
S_{d(\varepsilon)})$ with $d=d(\e)$. Due to the second 
inequality in Lemma 1 
with $\varepsilon\in(0,1)$ we have  
\begin{eqnarray*}
j(\varepsilon)^{d(\varepsilon)} \le   
n(\varepsilon^{2 d(\varepsilon)},S_{d(\varepsilon),\bsgamma}) 
\le  C_{\delta} \left[1+\log 
\left(\frac{1}{\varepsilon}\right)^{2 d(\varepsilon)}\right]^{p^*+\delta}.
\end{eqnarray*}
This yields
$$j(\varepsilon)^{d(\varepsilon)} \le C_{\delta}\left[(1+2
  d(\varepsilon)) \log \frac{1}{\varepsilon}\right]^{p^*+\delta},
$$ 
and hence 
\begin{eqnarray*}
d(\varepsilon) 
\log j(\varepsilon) &\le&  
(p^*+\delta) \left[\log(1+2 d(\varepsilon)) +  
\log \log\frac{1}{\varepsilon}\right]+\log C_{\delta}\nonumber\\
& \le & (p^*+\delta) \left[\log d(\varepsilon) +  
\log \log\frac{1}{\varepsilon}\right]+(p^*+\delta) \log 3 +\log C_{\delta}.
\end{eqnarray*}
Now, since $d(\varepsilon) \rightarrow \infty$ and $\log\log
\frac{1}{\varepsilon}\rightarrow \infty$ as $\varepsilon\rightarrow 0$
it follows that 
\begin{eqnarray}\label{est1}
\limsup_{\varepsilon \rightarrow 0} \frac{d(\varepsilon)\log
  j(\varepsilon)}{\log d(\varepsilon)+\log \log 
\frac{1}{\varepsilon}} \le p^*+\delta.
\end{eqnarray}
For $\varepsilon \rightarrow 0$,
we have $$\frac{d(\varepsilon)}{\log d(\varepsilon)} \rightarrow
\infty$$ 
and $\log j(\varepsilon) \rightarrow \infty$ or  
$\log j(\varepsilon) \rightarrow \log k$, with $k\ge2$,
where the latter case 
appears when $\lambda_1 \ge \lambda_2 \ge \ldots \ge \lambda_k>0$ 
and $\lambda_j = 0$ for $j \ge k+1$. This means that we have  
$$
p^*+\delta \ge \limsup_{\varepsilon \rightarrow 0}
\underbrace{\frac{d(\varepsilon)}{\log d(\varepsilon)}}_{\rightarrow
  \infty} \frac{\overbrace{\log j(\varepsilon)}^{\rightarrow \
    \mbox{{\scriptsize $\infty$ or $\log k$}}} }{1+\frac{\log \log
    \frac{1}{\varepsilon}}{\log d(\varepsilon)}}.
$$ 
    
From this we
deduce that 
$$
\lim_{\varepsilon \rightarrow 0} \frac{\log \log
  \frac{1}{\varepsilon}}{\log d(\varepsilon)} =\infty\ \ \mbox{or,
  equivalently, }\ \ \log d(\varepsilon) = o\left(\log \log
  \frac{1}{\varepsilon}\right).
$$ 
Hence $$\log d(\varepsilon)+\log
\log \frac{1}{\varepsilon}=(1+o(1)) \log \log \frac{1}{\varepsilon}$$
and, inserting this into \eqref{est1}, $$\limsup_{\varepsilon
  \rightarrow 0} \frac{ d(\varepsilon) \log j(\varepsilon)}{\log \log
  \frac{1}{\varepsilon}} \le p^* + \delta.
$$ 
Since $\delta>0$ can be arbitrarily small, this means that $B \le
p^*$ as needed. 
Therefore, the proof for EXP-SPT is completed with $p^*=B$.

\item We only need to show that EXP-PT implies EXP-SPT. 
The conditions $\lim_{j \rightarrow \infty}\lambda_j=\lim_{j \rightarrow \infty}\gamma_j=0$ can be shown as
before.
Under the assumption of EXP-PT and Lemma~\ref{le1}  with $d=d(\e)$, 
there exist non-negative numbers 
$C, q$, and $\tau$ such that $$j(\varepsilon)^{d(\varepsilon)} 
\le n(\varepsilon^{2 d(\varepsilon)},S_{d(\varepsilon),\bsgamma}) 
\le C d(\varepsilon)^q \left(1+2 d(\varepsilon) 
\log \frac{1}{\varepsilon}\right)^{\tau} 
\le C_1 d(\varepsilon)^{q+\tau} 
\left(\log \frac{1}{\varepsilon}\right)^{\tau},
$$ 
where in the last estimate we assumed that $\varepsilon \le
1/{\rm e}$, and $C_1:=3^{\tau} C$. 
Taking the logarithm yields 
\begin{equation}\label{est2}
d(\varepsilon) \log j(\varepsilon) \le (q+\tau) \log d(\varepsilon)+
\tau 
\log \log \frac{1}{\varepsilon}+\log C_1.
\end{equation}
This shows that $$\limsup_{\varepsilon \rightarrow 0}
 \frac{d(\varepsilon) \log j(\varepsilon)}{\log d(\varepsilon)+\log
 \log \frac{1}{\varepsilon}} \le q+\tau.$$

 Now we argue as in the first part of this proof. We have 
$$
\limsup_{\varepsilon \rightarrow 0} 
\underbrace{\frac{d(\varepsilon)}{\log d(\varepsilon)}}_{\rightarrow
  \infty} 
\frac{\overbrace{\log j(\varepsilon)}^{\rightarrow \ 
\mbox{{\scriptsize $\infty$ or $\log k$}}} }{1+\frac{\log \log \frac{1}{\varepsilon}}{\log d(\varepsilon)}} \le q+\tau$$ and hence $$\lim_{\varepsilon \rightarrow 0} \frac{\log \log \frac{1}{\varepsilon}}{\log d(\varepsilon)} =\infty\ \ \mbox{ or, equivalently,}\ \ \log d(\varepsilon) = o\left(\log \log \frac{1}{\varepsilon}\right).$$
Going back to \eqref{est2} we obtain $$d(\varepsilon) \log j(\varepsilon) \le \tau \log \log \frac{1}{\varepsilon}+o\left(\log \log \frac{1}{\varepsilon}\right)$$ and hence $B_{{\rm EXP-SPT}} \le \tau$. This means that we have EXP-SPT and we are done.
\item Assume first that we have 
EXP-QPT with the exponent $t$, i.e., for all $\varepsilon\in(0,1)$
and $d\in\NN$ we have
$$
n(\varepsilon,S_{d,\bsgamma})\le
C\exp\left(t(1+\log\,d)(1+\log(1+\log \varepsilon^{-1}))\right).
$$
 
Then $\lim_{j \rightarrow \infty}\lambda_j=0$ follows from the fact that
$n(\varepsilon,S_{d,\bsgamma})$ is finite.
Using again \eqref{gammazero} we conclude that $\lim_{j \rightarrow \infty}\gamma_j=0$. 
Indeed, if $\lim_{j \rightarrow \infty}\gamma_j=\gamma^*>0$ we have
$$
d\,\log\,2\le \log\,n(\tfrac12(\gamma^*\lambda_2)^{d/2},S_{d,\bsgamma})=
\mathcal{O}(t (\log d)^2),
$$
which is a contradiction for large $d$. 

We now apply Lemma~\ref{le1} for $\varepsilon \in
(0,1/{\rm e})$ and $d=d(\e)\ge1$. 
From the definition of EXP-QPT 
we obtain
\begin{eqnarray*}
j(\varepsilon)^{d(\varepsilon)} &\le& 
C \exp\left(t(1+\log d(\varepsilon))
\left(1+ \log\left(1+ 
\log \left(\frac{1}{\varepsilon}\right)^{2 d(\varepsilon)}\right)
\right)\right)\\
& \le & C \exp\left(t(1+\log d(\varepsilon))
\left(1+2\log 2+ \log d(\varepsilon)+ 
\log \log \frac{1}{\varepsilon}\right)\right).
\end{eqnarray*}
Taking the logarithm yields 
\begin{equation}\label{est3}
d(\varepsilon) \log j(\varepsilon) \le 
\log C + t(1+\log d(\varepsilon))\left(1+2\log 2+
  \log d(\varepsilon)
+ \log \log \frac{1}{\varepsilon}\right)
\end{equation}
and hence $$\limsup_{\varepsilon \rightarrow 0} \frac{d(\varepsilon)
  \log j(\varepsilon)}{\log(1+d(\varepsilon)) [\log \log
  \frac{1}{\varepsilon} + \log d(\varepsilon)] } \le t.
$$ 
We use, at this stage, the already familiar argument:
$$\limsup_{\varepsilon \rightarrow 0} 
\underbrace{\frac{d(\varepsilon)}{\log d(\varepsilon) 
\log (1+d(\varepsilon))} }_{\rightarrow \infty} 
\frac{\overbrace{\log j(\varepsilon)}^{\rightarrow \ 
\mbox{{\scriptsize $\infty$ 
or $\log k$}}}}{1+ \frac{\log \log
\frac{1}{\varepsilon}}{\log d(\varepsilon)}}\le t
$$ 
and hence 
$$
\log d(\varepsilon)=o\left(\log \log \frac{1}{\varepsilon}\right).
$$ 
Inserting this into \eqref{est3} we obtain that 
\begin{equation}\label{eq:BEXPQPT}
B_{{\rm EXP-QPT}}:=\limsup_{\varepsilon \rightarrow 0} 
\frac{d(\varepsilon) \log j(\varepsilon)}{[\log
  d(\varepsilon)]\log\log 
\frac{1}{\varepsilon}}\le t,
\end{equation}
as needed.
\vskip 1pc
Now assume that $\lim_{j \rightarrow \infty}\lambda_j=\lim_{j \rightarrow \infty}\gamma_j=0$ and that
$B:=B_{{\rm EXP-QPT}}<\infty$. 
Then for every $\delta>0$ there exists an
$\varepsilon_{\delta}>0$ such that 
\begin{equation}\label{condBQPT}
\frac{d(\varepsilon) 
\log j(\varepsilon)}{[\log d(\varepsilon)]
\log\log \frac{1}{\varepsilon}}\le B +\delta \ \ \ 
\mbox{for all $\varepsilon \in (0,\varepsilon_{\delta})$.}
\end{equation}
We further assume that $\varepsilon_{\delta}$ is small enough
such that 
$d(\varepsilon)\ge 2$ for all $\varepsilon \in (0,\varepsilon_{\delta})$.

{}From this and Lemma~\ref{le1} we obtain  
for all $\varepsilon \in (0,\varepsilon_{\delta})$, 
$$
n(\varepsilon,
S_{d,\bsgamma}) 
\le j(\varepsilon)^{d(\varepsilon)} \le 
\exp\left((B+\delta) [\log d(\varepsilon)] 
\log \log \frac{1}{\varepsilon}\right).
$$ 
Hence, for all $\varepsilon \in (0,\varepsilon_{\delta})$ 
and all $d \ge d(\varepsilon)$ we have 
$$
n(\varepsilon,S_{d,\bsgamma}) \le 
\exp\left((B+\delta) [\log d] \log \log \frac{1}{\varepsilon}\right).
$$
If $3\le d \le d(\varepsilon)$ then we obtain from \eqref{condBQPT} 
for all $\varepsilon \in (0,\varepsilon_{\delta})$ 
$$
\log j(\varepsilon)\le (B+\delta)  
\frac{\log d(\varepsilon)}{d(\varepsilon)} 
\log \log \frac{1}{\varepsilon} \le 
(B+\delta) \frac{\log d}{d}  \log \log \frac{1}{\varepsilon},
$$ 
where in the last estimate we use the fact that $d(\varepsilon)\ge 3$
and that the function $d \mapsto (\log d)/d$ is decreasing for
$d\in\{3,\ldots,d(\varepsilon)\}$. 

For $d=2$ we obtain again from \eqref{condBQPT} 
for all $\varepsilon \in (0,\varepsilon_{\delta})$, 
$$
\log j(\varepsilon)\le (B+\delta)  
\frac{\log d(\varepsilon)}{d(\varepsilon)} 
\log \log \frac{1}{\varepsilon} \le 
(B+\delta) \frac{1 + \log 2}{2}  \log \log \frac{1}{\varepsilon},
$$ 
and for $d=1$ we obtain from \eqref{condBQPT} for all $\varepsilon \in (0,\varepsilon_{\delta})$, 
$$
\log j(\varepsilon)\le (B+\delta)  
\frac{\log d(\varepsilon)}{d(\varepsilon)} 
\log \log \frac{1}{\varepsilon} \le 
(B+\delta)  \log \log \frac{1}{\varepsilon}.
$$

Hence we obtain $$\log j(\varepsilon)^d\le  (B+\delta)  [1+\log d] \log \log \frac{1}{\varepsilon}.$$ 

The first estimate in Lemma~\ref{le1} tells us that 
$$
n(\varepsilon,S_{d,\bsgamma}) \le 
j(\varepsilon)^{\min(d(\varepsilon),d)}.
$$ 
Hence, for all $\varepsilon \in (0,\varepsilon_{\delta})$ 
and all $d\in\{1,\ldots,d(\varepsilon)\}$ we have
\begin{eqnarray*}
 n(\varepsilon,S_{d,\bsgamma}) \le j(\varepsilon)^d \le \exp\left((B+\delta) [1+\log d] \log\log \frac{1}{\varepsilon} \right).
\end{eqnarray*}

This implies $$n(\varepsilon,S_{d,\bsgamma}) 
\le \exp\left((B+\delta) [1+\log d] 
\log\log \frac{1}{\varepsilon} \right)
$$ 
for all $\varepsilon \in (0,\varepsilon_{\delta})$ 
and for all $d \in \NN$. 

Finally, for $\e\in(\e_\delta,1)$ we set
$C_\delta=n(\e_\delta,S_{d(\e_\delta),\bsgamma})$ and conclude
$$
n(\e,S_{d,\bsgamma})\le n(\e_\delta,S_{d(\e_{\delta}),\bsgamma})
\le C_\delta\exp\left((B+\delta)[\log d]\log \log \frac1{\e}\right).
$$ 
Hence there exists a $C_{\delta}>0$ such that 
$$
n(\varepsilon,S_{d,\bsgamma}) \le 
C_{\delta} \exp\left((B+\delta) 
[1+\log d] \log \log \frac{1}{\varepsilon}\right)
$$ 
for all $\varepsilon \in (0,1)$ and $d \in \NN$. 
This implies EXP-QPT
with $t\le B+\delta$. Since $t\ge B$ due to \eqref{eq:BEXPQPT}
and the positive $\delta$ can be arbitrarily small,  
the infimum of such $t$ is $B$, as claimed.

\item From \cite[Theorem~3]{KW19} 
we know that EXP-$(s,t)$-WT holds for any positive $s$ and $t$ 
if and only if 
\begin{equation}\label{condWT}
\sup_{d \in \NN}\exp(-c d^t)\, 
\sum_{j=1}^{\infty}{\rm e}^{-c \left(1+
\log \frac{2}{\lambda_{d,\bsgamma,j}}\right)^s}< \infty \ \ \ \mbox{ for all $c>0$,}
\end{equation}
where $\lambda_{d,\bsgamma,1} \ge \lambda_{d,\bsgamma,2} \ge \ldots$ denote the eigenvalues of 
$W_{d,\bsgamma}= S_{d,\bsgamma}^* S_{d,\bsgamma}: H_{d,\bsgamma}\rightarrow H_{d,\bsgamma}$ ordered in a non-increasing fashion.

Assume first that $s=1$. For the above sum we have
\begin{eqnarray*}
\sum_{j=1}^{\infty}{\rm e}^{-c \left(1+\log \frac{2}{\lambda_{d,\bsgamma,j}}\right)} & = & \frac{1}{{\rm e}^{c(1+\log 2)}}  \sum_{j_1,\ldots,j_d=1}^{\infty} {\rm e}^{c \log(\lambda_{1,j_1}\lambda_{2,j_2}\cdots \lambda_{d,j_d})}\\
& = & \frac{1}{(2 {\rm e})^c} \prod_{k=1}^d \left(\sum_{j=1}^{\infty} \lambda_{k,j}^c\right) =  \frac{1}{(2 {\rm e})^c} \prod_{k=1}^d \left(1+\gamma_k^c\sum_{j=2}^{\infty} \lambda_j^c\right).
\end{eqnarray*}
Hence, EXP-$(1,t)$-WT is equivalent to 
\begin{equation*}\label{condWT2}
\sup_{d \in \NN}\exp(-c d^t)\, 
\prod_{k=1}^d \left(1+\gamma_k^c\sum_{j=2}^{\infty} \lambda_j^c\right)< \infty \ \ \ \mbox{ for all $c>0$.}
\end{equation*}
Taking the logarithm, we find that EXP-$(1,t)$-WT is equivalent to   
\begin{equation}\label{condWT3}
\sup_{d \in \NN}\left(\sum_{k=1}^d 
\log \left(1+\gamma_k^c\sum_{j=2}^{\infty} \lambda_j^c\right) - 
c d^t\right)< \infty \ \ \ \mbox{ for all $c>0$.}
\end{equation}

Assume first that \eqref{condWT3} holds. Then we have 
$\sum_{j=2}^{\infty} \lambda_j^c < \infty$ for all $c>0$ 
and hence, by Lemma~\ref{le2}, 
$$
\lim_{j \rightarrow \infty}\frac{\log \frac{1}{\lambda_j}}{\log
  j}=\infty.
$$ 
Consider the case $t=1$. Assume that $\lim_{j \rightarrow \infty}\gamma_j=\gamma_*>0$.
Then \eqref{condWT3} implies 
$$
\sup_{d \in \NN}\left(\log \left(1+
\gamma_*^c\sum_{j=2}^{\infty}
    \lambda_j^c\right) - c\right) d < \infty \ \ \ \mbox{ for all
  $c>0$.}
$$ 
This, however, yields a contradiction, since 
$\gamma_*^c\sum_{j=2}^{\infty}
    \lambda_j^c$ tends to infinity with $c$ approaching zero,
and therefore for small $c>0$ we have 
$$\log \left(1+\gamma_*^c\sum_{j=2}^{\infty} \lambda_j^c\right) - c
 >0.$$
Hence, we must have 
$$
 \lim_{j \rightarrow \infty} \gamma_j=0.
$$ 
Thus we have shown the necessary conditions for EXP-WT.

Now assume that 
$$
\lim_{k \rightarrow \infty} \gamma_k=0\ \ \ \mbox{ and } 
\lim_{j \rightarrow \infty}\frac{\log \frac{1}{\lambda_j}}{\log
  j}=\infty.
$$ 
{}From Lemma~\ref{le2} we obtain that  
$M_c^*:=\sum_{j=2}^{\infty} \lambda_j^c < \infty$ for all $c>0$. 
Hence, for every fixed $c>0$ we have
\begin{eqnarray*}
\sum_{k=1}^d \log \left(1+\gamma_k^c\sum_{j=2}^{\infty} \lambda_j^c\right) - c d & \le &  M_c^* \sum_{k=1}^d \gamma_k^c  - c d\\
& = & M_c^* \sum_{k=1}^d \left( \gamma_k^c -\frac{c}{M_c^*}\right)\\
& \le & M_c^* \sum_{k=1}^{k_c^*} \left( \gamma_k^c -\frac{c}{M_c^*}\right)< \infty,
\end{eqnarray*}
where $k_c^*$ is the largest $k \in \NN$ such that $\gamma_k^c -\frac{c}{M_c^*}>0$. This number is well defined, since $\lim_{k \rightarrow \infty} \gamma_k=0$. Hence \eqref{condWT3} holds for $t=1$ and this implies EXP-WT.
\item Let $s=1$ and $t<1$. Assume we have EXP-$(1,t)$-WT. 
Then we have EXP-WT and hence 
$$
\lim_{k\rightarrow \infty}\gamma_k=0 \ \ \mbox{and}\ \ \ 
\lim_{j \rightarrow \infty} \frac{\log \frac{1}{\lambda_j}}{\log
  j}=\infty.
$$ 
Now \eqref{condWT3} can be rewritten as 
$$
\sum_{k=1}^d \log \left(1+\gamma_k^c M_c^*\right) - c d^t \le A_c\ \ \
\mbox{ for all $d\in \NN$},$$ with a positive $A_c<\infty$ for every
$c>0$. 
We use the inequality
$$
\log(1+x) \ge (\log 2)\,x\ \ \ \mbox{for all $x \in [0,1]$}. 
$$
Let $k_c^*$ be the largest $k$ such that $\gamma_k^c M_c^*>1$. If such a $k_c^*$ does not exist, we set $k_c^*:=0$. Then we have
\begin{eqnarray*}
A_c & \ge & \sum_{k=1}^d \log \left(1+\gamma_k^c M_c^*\right) - c d^t\\
& \ge &  \sum_{k=1}^{k_c^*} \log \left(1+\gamma_k^c M_c^*\right) +  (\log 2) M_c^* \sum_{k=k_c^*+1}^d \gamma_k^c  - c d^t\\
& \ge &  \sum_{k=1}^{k_c^*} \log \left(1+\gamma_k^c M_c^*\right) +  (\log 2) M_c^* (d-k_c^*) \gamma_d^c  - c d^t.
\end{eqnarray*}
{}From here it follows by an argument similar to the one used in the proof of Lemma~\ref{le2} that
$$\liminf_{d\rightarrow \infty} \frac{\log
  \frac{1}{\gamma_d}}{\log d} \ge \frac{1-t}{c} \ \ \ \mbox{for all
  $c>0$.}
$$ Since $c$ can be arbitrarily small and $t<1$ we have
$$\lim_{d\rightarrow \infty} \frac{\log
  \frac{1}{\gamma_d}}{\log d}=\infty,
$$ as desired.

On the other hand for
$$
\lim_{j \rightarrow \infty} 
\frac{\log \frac{1}{\gamma_j}}{\log j}=\infty\ \ 
\mbox{ and }\ \ \lim_{j \rightarrow \infty} \frac{\log 
\frac{1}{\lambda_j}}{\log j}=\infty,
$$ 
Lemma~\ref{le2} yields  
$$
\Gamma_c:=\sum_{k=1}^{\infty} \gamma_k^c < \infty \ \ \mbox{and}\ \  M_c^*:=\sum_{j=2}^{\infty} \lambda_j^c < \infty \ \ \ \mbox{for all $c>0$.}$$ Now, for every fixed $c>0$ we have
\begin{eqnarray*}
\sum_{k=1}^d \log \left(1+\gamma_k^c\sum_{j=2}^{\infty} \lambda_j^c\right) - c d^t & \le &  M_c^* \Gamma_c - c d^t
\end{eqnarray*}
and hence $$\sup_{d \in \NN} \sum_{k=1}^d \log
\left(1+\gamma_k^c\sum_{j=2}^{\infty} \lambda_j^c\right) - c d^t <
\infty.$$ Hence \eqref{condWT3} holds 
and this implies EXP-$(1,t)$-WT.
\item Let $s=1$ and $t>1$. 
    
The sufficiency of the condition on the eigenvalues is shown for the un-weighted case (i.e., all $\gamma_j$ equal 1) in  Case (N.1) of \cite[Theorem~2]{HKW19}. If some of the weights satisfy $\gamma_j < 1$, then the problem is easier than in the un-weighted case, which means that the sufficient condition holds as well.

Regarding necessity of the condition, we have shown that EXP-$(1,t)$-WT implies \eqref{condWT3} and hence we find as above that the condition $$\lim_{j \rightarrow \infty} \frac{\log \frac{1}{\lambda_j}}{\log j}=\infty$$ is indeed necessary.

\item Let $s > 1$, $t\le1$ and $\lambda_2<1$.     
The sufficiency of the condition on the eigenvalues is shown for the un-weighted case (i.e., all $\gamma_j$ equal 1) in  Case (N.3) of \cite[Theorem~2]{HKW19}. If some of the weights satisfy $\gamma_j < 1$, then the problem is easier than in the un-weighted case, which means that the sufficient condition holds as well.

We show that the condition on the $\lambda_j$'s is also necessary. We use \eqref{condWT} and Lemma~\ref{le3} to obtain 
\begin{eqnarray*}
\lefteqn{\sum_{j=1}^{\infty}{\rm e}^{-c \left(1+\log \frac{2}{\lambda_{d,\bsgamma,j}}\right)^s}}\\
& \ge & {\rm e}^{-c 2^{s-1} (1+\log 2)^s} \sum_{j=1}^{\infty}  {\rm e}^{-c 2^{s-1} \left(\log \frac{1}{\lambda_{d,\bsgamma,j}}\right)^s}\\
& = &  {\rm e}^{-c 2^{s-1} (1+\log 2)^s} \sum_{j_1,\ldots,j_d=1}^{\infty}  {\rm e}^{-c 2^{s-1} \left(\sum_{k=1}^d \log \frac{1}{\lambda_{k,j_k}}\right)^s}\\
& \ge &  {\rm e}^{-c 2^{s-1} (1+\log 2)^s} \sum_{j_1,\ldots,j_d=1}^{\infty}  {\rm e}^{-c (2d)^{s-1} \sum_{k=1}^d \left(\log \frac{1}{\lambda_{k,j_k}}\right)^s}\\
& = &  {\rm e}^{-c 2^{s-1} (1+\log 2)^s} \prod_{k=1}^d \left(1+\sum_{j=2}^{\infty}  {\rm e}^{-c (2d)^{s-1} \left(\log \frac{1}{\gamma_k} + \log \frac{1}{\lambda_j}\right)^s}\right)\\
& \ge & {\rm e}^{-c 2^{s-1} (1+\log 2)^s} \prod_{k=1}^d \left(1+  {\rm e}^{-c (4 d)^{s-1} \left(\log \frac{1}{\gamma_k}\right)^s} \sum_{j=2}^{\infty}  {\rm e}^{-c (4 d)^{s-1} \left(\log \frac{1}{\lambda_j}\right)^s}\right).
\end{eqnarray*}
Put
\begin{equation}\label{defGL}
\Gamma_k:={\rm e}^{-(\log \frac{1}{\gamma_k})^s}\ \ \ \mbox{ and }\ \ \ \Lambda_j:={\rm e}^{- (\log \frac{1}{\lambda_j})^s}.
\end{equation}
Then we have $$\sum_{j=1}^{\infty}{\rm e}^{-c \left(1+\log \frac{2}{\lambda_{d,\bsgamma,j}}\right)^s} \ge {\rm e}^{-c 2^{s-1} (1+\log 2)^s} \prod_{k=1}^d \left(1+  \Gamma_k^{c (4 d)^{s-1}} \sum_{j=2}^{\infty}  \Lambda_j^{c (4 d)^{s-1}} \right).$$

Assume that EXP-$(s,t)$-WT holds true. Then according to \eqref{condWT}  together with the above lower bound (for $d=1$) we obtain 
$$\exp(-c (1+ 2^{s-1} (1+\log 2)^s)) \left(1+  \Gamma_1^{c 4^{s-1}} \sum_{j=2}^{\infty}  \Lambda_j^{c 4^{s-1}} \right) < \infty\ \ \ \mbox{ for all $c>0$.}$$
This requires that  $$\sum_{j=2}^{\infty}  \Lambda_j^{c} < \infty\ \ \ \mbox{ for all $c>0$.}$$ According to Lemma~\ref{le2} this is equivalent to 
$$
\lim_{j \rightarrow \infty} \frac{\log \frac{1}{\Lambda_j}}{\log j}=\infty,
$$ 
and this condition holds if and only if 
\begin{equation*}
\lim_{j \rightarrow \infty} \frac{\left(\log \frac{1}{\lambda_j}\right)^s}{\log j}=\infty.
\end{equation*}
This finishes the proof for the necessary condition. Note that for this part we did not use that $\lambda_2<1$.

\item Let $s>1$, $t\le1$ and $\lambda_2=1$.
The necessary condition on the $\lambda_j$'s follows from the above where we did not use that $\lambda_2<1$.
To show that EXP-$(s,t)$-WT implies $\gamma_p<1$ we apply Lemma
\ref{le4} with all $k_j=2$ for which $\lambda_2=1$.
If all $\gamma_j=1$ then $\lim_{d\to\infty}d^t/(d\log 2)$ is zero
for $t<1$ and $1/\log 2$ for $t=1$ and it is never infinity. 
It is infinity only if $\gamma_p<1$ for some $p$, as claimed. 

In order to show that the conditions on the $\lambda_j$'s and $\gamma_j$'s
imply EXP-$(s,t)$-WT, we switch to a possibly harder problem for which
the weights are given by
$$
1=\gamma_1=\ldots=\gamma_{p-1}>\gamma_p=\gamma_{p+1}=\ldots.
$$
and the eigenvalues
$$
\widetilde{\lambda}_1=\ldots=\widetilde{\lambda}_{p-1}=1\ \ \ \mbox{and}\ \ \ \widetilde{\lambda}_j=\gamma_p
\lambda_j \  \ \mbox{for $j=p,p+1,\dots$}.
$$  
Note that $\widetilde{\lambda}_p<\widetilde{\lambda}_{p-1}=1$. 

For $d\ge p$, we have
\begin{eqnarray*}
\lefteqn{
\sum_{j=1}^\infty {\rm e}^{-c\left(1+\log \frac2{\lambda_{d,\bsgamma,j}}
\right)^s}} \\
&\le&\sum_{j_1,\dots,j_d=1}^\infty
{\rm e}^{-c\left(\sum_{k=1}^d\log\frac1{\lambda_{k,j_k}}\right)^s}\\
&\le&
\sum_{j_1,\dots,j_{p-1}=1}^\infty
{\rm e}^{-c\sum_{k=1}^{p-1}\left(\log\frac1{\lambda_{k,j_k}}\right)^s}
\sum_{j_p,\dots,j_d=1}^\infty
{\rm e}^{-c\left(\sum_{k=p}^d\log\frac1{\widetilde\lambda_{j_k}}\right)^s}\\
&\le&\left(1+\sum_{j=2}^\infty\Lambda_j^c\right)^{p-1}
\sum_{j_p,\dots,j_d=1}^\infty
{\rm e}^{-c\left(\sum_{k=p}^d\log\frac1{\widetilde\lambda_{j_k}}\right)^s},
\end{eqnarray*}
where $\Lambda_j={\rm e}^{-\left(\log\frac1{\lambda_j}\right)^s}$.

The series $\sum_{j=2}^\infty \Lambda_j^c$ is convergent due to Lemma \ref{le2}
and the conditions on the $\lambda_j$'s. Therefore the first factor
is of order $1$. 
The second factor
$$
\sum_{j_p,\dots,j_d=1}^\infty
{\rm e}^{-c\left(\sum_{k=p}^d\log\frac1{\widetilde\lambda_{j_k}}\right)^s}
=\exp\left(\Theta\left(\left(\log d\right)^{\max(1,s/(s-1)}\right)\right),
$$
as proved in \cite[(A.4)]{HKW19}. Therefore
$$
\sup_d {\rm e}^{-cd^t} 
\sum_{j=1}^\infty {\rm e}^{-c\left(1+\log \frac2{\lambda_{d,\bsgamma,j}}
\right)^s}=
\sup_d\exp\left(-cd^t+\Theta\left((\log
    d)^{\max(1,s/(s-1))}\right)\right)<\infty,
$$
which completes the proof of this item. 
\item
Let $s>1$ and $t>1$. The sufficiency of the condition on the eigenvalues is shown for the un-weighted case (i.e., all $\gamma_j$ equal 1) in  Case (N.1) of \cite[Theorem~2]{HKW19}. If some of the weights satisfy $\gamma_j < 1$, then the problem is easier than in the un-weighted case, which means that the sufficient condition holds as well. 

In order to show that the condition on the $\lambda_j$'s is also necessary one proceeds as above in Item {\it 7.}

\item Let $s<1$ and $t>1$. Note that EXP-$(s,t)$-WT holds for arbitrary
$\gamma_j$'s iff this holds for $\gamma_j=1$ for all $j\in\NN$.
This case is proved in Case (N.2) of \cite[Theorem~2]{HKW19}.
\item
Let $s<1$ and $t=1$. 
Suppose first that EXP-$(s,1)$-WT holds. Then EXP-WT also holds and
$\lim_{k\rightarrow \infty}\gamma_k=\lim_{j \rightarrow \infty}\lambda_j=0$. 
Take integers $d,k,j$ with $j\ge2$ and $k\le d$, and define
$$
\e^{2}=\gamma_k^d\lambda_j^d\alpha,
$$
where $\alpha \in(0,1)$ such that $\e<1$.
Then
$\log n(\e,S_{d,\bsgamma})\ge d\log j$ and, proceeding as before, we conclude
\begin{eqnarray*}
\infty&=&
\lim_{d+\gamma_k^{-d}\lambda_j^{-d}\to\infty}
\frac{d+\left(\log\frac1{\e}\right)^s}{d \log j}\\
&=&
\lim_{d +\gamma_k^{-d}\lambda_j^{-d}\to\infty}
\frac{d+2^{-s}\left[d\left(\log
      \frac1{\gamma_k}\right)+d\left(\log \frac1{\lambda_j}\right)
+\log \frac1{\alpha}\right]^s}{d \log j}\\
&\le&
\lim_{d+\gamma_k^{-d}\lambda_j^{-d}\to\infty}
\frac{d+\left(\log \frac1{\alpha}\right)^s+d^s\left(\left(\log\frac1{\gamma_k}
\right)^s+\left(\log\frac1{\lambda_j}\right)^s\right)}{d \log j}.
\end{eqnarray*}
Since $[d+(\log \frac{1}{\alpha})^s]/(d \log j)$ does not go to
infinity, we obtain \eqref{111}.  

Suppose now that \eqref{111} holds. We show EXP-$(s,1)$-WT by using 
\eqref{condWT}. {}From Lemma~\ref{le3} we get
\begin{eqnarray*}
\alpha_d & := & \sum_{j=1}^\infty 
{\rm e}^{-c\left(1+\log\frac{2}{\lambda_{d,\bsgamma,j}}\right)^s}\\
& \le & \sum_{j_1,\dots,j_d=1}^\infty {\rm e}^{-c\left(\sum_{k=1}^d\log
\frac1{\lambda_{d,j_k}}\right)^s} \\
& \le &  
\sum_{j_1,\dots,j_d=1}^\infty {\rm e}^{-cd^{s-1}
\sum_{k=1}^d\left(\log\frac1{\lambda_{k,j_k}}\right)^s}.
\end{eqnarray*}
Then, again by Lemma~\ref{le3}, we obtain  
$$
\alpha_d
\le
\prod_{k=1}^d\left(1+\sum_{j=2}^\infty
{\rm e}^{-c(2d)^{s-1}\left[\left(\log
  \frac1{\gamma_k}\right)^s+\left(\log\frac1{\lambda_j}\right)^s\right]}\right).
$$
{}From \eqref{111} we conclude that for any (large) $M$ there exists $C_M$ such that 
for all $d\ge C_M$ we have 
$$
{\rm e}^{-c(2d)^{s-1}\left[\left(\log
  \frac1{\gamma_k}\right)^s+\left(\log\frac1{\lambda_j}\right)^s\right]}\le 
j^{-c2^{s-1}M}.
$$  
Note that the exponent $c2^{s-1}M$ can be sufficiently large
for large $M$, and
therefore the series $\sum_{j=2}^\infty j^{-c2^{s-1}M}$ is convergent
and sufficiently small, say it is $o(c)$.
Then (19) implies that 
\begin{eqnarray*}
\sup_{d\in\NN}
\exp(-c\,d)\alpha_d&=&\sup_{d\in\NN}
\exp(-c\,d)(1+o(c))^d\\
&=&\sup_{d\in\NN}\exp(d(-c+o(c)))<\infty.
\end{eqnarray*}
Hence, EXP-$(s,1)$-WT holds, and the proof is complete.
\end{enumerate}
\end{proof}

\begin{small}
\noindent\textbf{Authors' addresses:}
\\ \\
\noindent Peter Kritzer,
\\
Johann Radon Institute for Computational and Applied Mathematics (RICAM),
Austrian Academy of Sciences, Altenbergerstr.~69, 4040 Linz, Austria\\
 \\
\noindent Friedrich Pillichshammer,
\\
Department of Financial Mathematics and Applied Number Theory,
Johannes Kepler University Linz, Altenbergerstr.~69, 4040 Linz, Austria\\
 \\
\noindent Henryk Wo\'{z}niakowski, \\
Department of Computer Science, Columbia University, New York 10027,
USA and Institute of Applied Mathematics,
University of Warsaw, ul. Banacha 2, 02-097 Warszawa, Poland\\ \\

\noindent \textbf{E-mail:} \\
\texttt{peter.kritzer@oeaw.ac.at}\\
\texttt{friedrich.pillichshammer@jku.at} \\
\texttt{henryk@cs.columbia.edu}
\end{small}
\end{document}